\documentclass[12pt]{amsart}
\usepackage{amssymb, amscd, amsmath, amsthm, epsf, epsfig, latexsym, enumerate}

\newtheorem{theorem}{Theorem}
\newtheorem{lemma}[theorem]{Lemma}
\newtheorem{cor}[theorem]{Corollary}

\newtheorem*{Def}{Definition}

\begin{document}

\title{Seifert fibred knot manifolds}
 
\author{J.A.Hillman \and J.Howie}
\address{School of Mathematics and Statistics,
     University of Sydney,
     \newline
      NSW 2006,      Australia}
 \address{
Department of Mathematics and Maxwell Institute for Mathematical Sciences,
Heriot-Watt University, Edinburgh 
\newline
EH14 4AS, UK }

\email{jonathan.hillman@sydney.edu.au,j.howie@hw.ac.uk}

\begin{abstract}
We consider the question of when is the closed manifold
obtained by elementary surgery on an $n$-knot Seifert fibred 
over a 2-orbifold.
The possible bases are strongly constrained by the
fact that knot groups have weight 1.
We present a new family of 2-knots with solvable groups,
overlooked in earlier work.
The knots in this new family are neither invertible nor amphicheiral,
and the weight orbits for the knot groups are parametrized by $\mathbb{Z}$.
There are no known examples in higher dimensions.
\end{abstract}



\maketitle

The knot manifolds of the title are the closed manifolds $M(K)=\chi(K,0)$ 
obtained by elementary surgery on $n$-knots $K$ in $S^{n+2}$.
We assume that the surgery is $0$-framed in the classical case $n=1$.
A higher dimensional knot is largely determined by its knot manifold,
together with the conjugacy class of a meridian in the knot group.
In the classical case 
this is probably not true.
(In particular, 
conjugacy classes in $\pi_1(M(K))$ do not 
determine isotopy classes of simple closed curves in $M(K)$.)
However important invariants of $K$, 
such as the Alexander module and the Blanchfield pairing, 
may be calculated in terms of $M(K)$, 
and whether $K$ is trivial, fibred, 
slice or DNC can each be detected by corresponding properties of $M(K)$.
(See \cite{Ga,Le}.)
Thus these manifolds have a privileged role.

The possible base orbifolds form a restricted class,
since the orbifold fundamental group must be the normal closure 
of a single element.
We adapt and extend the ideas of \cite{Ho} to the study of this condition.
The cases left open by our results seem unlikely.

Our main interest is in which 2-knot manifolds are Seifert fibred,
but we shall  consider briefly the other dimensions.
In the classical case there are also non-trivial Dehn surgeries,
parametrized by $\mathbb{Q}$.
If $K=K_{m,n}$ is the $(m,n)$-torus knot then the 3-manifolds 
$\chi(K,\frac{q}p)$ obtained by Dehn surgeries on $K$ are all Seifert fibred, 
with the sole exception of $\chi(K,\frac{-1}{mn})$,
which is the sum of two lens spaces \cite{Mo}.
For most other knots only finitely many Dehn surgeries give
Seifert manifolds, and much work has been done on 
establishing tight bounds on the 
numbers of such ``exceptional" Dehn surgeries for hyperbolic knots.
Our point of view is different, in that we concentrate on the case $\frac{p}q=0$.

In the course of constructing examples of 2-knot manifolds 
which are Seifert fibred over flat bases, 
we have discovered a family of knots with torsion-free polycyclic groups 
that were overlooked in an old result of
ours (Theorem  6.11 of \cite{Hi89}).
This gap does not materially affect much other work,
except that the claim in \cite{Hi11} 
that the classification of such knots is complete was unjustified.
We shall show here that none of the new knots is invertible or amphicheiral,
and the weight orbits for each knot group are parametrized by $\mathbb{Z}$.
Only the question of reflexivity remains undecided for these knots.

We know of no higher-dimensional examples, 
even if we allow the general fibre to be an 
orientable infrasolvmanifold of dimension $n$.

\section{some necessary conditions}

Let $K$ be a $n$-knot with knot group $\pi{K}$,
and let $\pi=\pi_1(M(K))$.
If $n=1$ then $\pi\cong\pi{K}/\langle\langle\lambda\rangle\rangle$ 
is the quotient of $\pi{K}$ by the normal closure of a longitude $\lambda$,
while if $n>1$ then $\pi\cong\pi{K}$.

Suppose that $M(K)$ is Seifert fibred over an aspherical 2-orbifold $B$.
Then $\pi$ is an extension of $\beta=\pi^{orb}(B)$ 
by a torsion free, 
virtually poly-$\mathbb{Z}$ group of Hirsch length $n$
and orientable type.
We shall show that the possible bases are strongly constrained
by the fact that quotients of knot groups have weight 1.
Much of this section is based on or extends ideas from  \cite{Ho}.

\begin{Def}
Let $\psi(x)$ be the distance to the nearest integer,
for $x\in\mathbb{R}$. Thus 
\[\psi(x)=x-\lfloor{x}\rfloor\quad\mathrm{ if}\quad{x-\lfloor{x}\rfloor\leq\frac12},
\]
while \[
\psi(x)=1-x+\lfloor{x}\rfloor\quad\mathrm{if}\quad{x-\lfloor{x}\rfloor>\frac12}.
\]
A triple $(\xi,\eta,\zeta)\in\mathbb{R}^3$ is {\it good\/} if
\[
2\max\{\psi(\xi),\psi(\eta),\psi(\zeta)\}<\psi(\xi)+\psi(\eta)+\psi(\zeta).
\]
\end{Def}

We shall use this concept to examine self-maps of $S^0=\{\pm1\}$,
which in turn can be used to determine the homotopy type 
of $S^1$-equivariant self-maps of $S^2$ (as in \cite{Ho}).

\begin{lemma}
Let $(\xi,\eta,\zeta)$ be a good triple of non-integer real numbers,
such that none of $\pm\xi\pm\eta\pm\zeta$ is an integer.
Define $\phi,\theta:S^0\to{S^0}$ by 
\[
\phi(\varepsilon)=(-1)^{\lfloor\xi+\eta+\varepsilon\zeta\rfloor}\quad\mathrm{and}\quad
\theta(\varepsilon)=(-1)^{\lfloor\xi-\eta+\varepsilon\zeta\rfloor}.
\]
Then $\phi\not=\theta$, and at least one of $\phi$, $\theta$ is a bijection.
\end{lemma}

\begin{proof}
Observe first that  adding an integer $n$ to any one of $\xi, \eta, \zeta$ has the effect 
of multiplying each of $\phi$, $\theta$ by $(-1)^n$.
Replacing $\eta$ by $-\eta$ has the effect of interchanging $\phi$ and $\theta$.
Replacing $\zeta$ by $-\zeta$ is equivalent to replacing $\varepsilon$ by $-\varepsilon$.
Finally, multiplying all three of $\xi,\eta,\zeta$ by $-1$ has the effect of changing 
the signs of $\phi$ and $\theta$. 
In each case, the problem remains the same.
Thus we may reduce to the case where $\xi,\eta,\zeta\in[0,\frac12]$.

Let us assume this is so.
By the definition of good, we see that $-1<\xi-\eta-\zeta<0$,
while $0<\xi-\eta+\zeta<1$, so $\theta$ is the identity map on $S^0$.
Similarly, $0<\xi+\eta-\zeta<1$, so that $\phi(-1)=+1$, and hence $\phi\not=\theta$.
\end{proof}

For convenience in what follows, we shall say that a positive integer $n$ is a 
{\it quasi-prime\/} if $n=4$ or $n$ is an odd prime.

\begin{lemma}
Let $a,b,c$ be distinct quasi-primes, and let $d,e,f\in\mathbb{Z}$ be integers 
coprime to $2a,2b,2c$, respectively. 
Assume that $\{a,b,c\}\not=\{3,4,5\}$.
Then there are positive integers $r,s,t$ which are coprime to $a,b,c$, respectively, 
and such that
\[
\frac{r}a+\frac{s}b+\frac{t}c<1
\quad{and}\quad(\frac{rd}{2a},\frac{se}{2b},\frac{tf}{2c})\quad{ is~good}.
\]
\end{lemma}

\begin{proof}
Since $\psi(x)=\psi(-x)=\psi(x+n)$, for all $x\in\mathbb{R}$ and $n\in\mathbb{Z}$,
we may assume without loss of generality that $0<d<a$, $0<e<b$ and $0<f<c$.
Let $\alpha=\frac{d}{2a}$, $\beta=\frac{e}{2b}$ and $\gamma=\frac{f}{2c}$.
Relabelling, if necessary, we may assume that $\alpha<\beta<\gamma$.

Define $u=\min\{\rho\in\mathbb{Z}:\rho\alpha+\beta>\gamma\}$. Then $u\geq1$ and 
\begin{equation}
u\alpha=(u-1)\alpha+\alpha<\gamma-\beta+\alpha<\gamma<\frac12,
\end{equation}
so $\psi(u\alpha)=u\alpha$ and $(u\alpha,\beta,\gamma)$ is good.
The result follows with $(r,s,t)=(u,1,1)$ unless
\begin{equation}
\frac{u}a+\frac1b+\frac1c\geq1.
\end{equation}
Assume therefore that inequality (2) holds.

If $d>1$ then $d\geq3$ since it is odd. Hence $3u<a$ by (1), whence (2) gives
\[
1\leq\frac13+\frac1b+\frac1c,
\]
contrary to hypothesis. Hence $d=1$.

Next note that $f\leq{c}-2$ since $f$ is odd and $\gamma<\frac12$.
Also $u\alpha<\gamma-\beta+\alpha$ by (1), and so
\[
\frac{u}a=2u\alpha\leq2\gamma-2\beta+2\alpha\leq1-\frac2c-\frac{e}b+\frac1a.
\]
Combining this with (2) gives
\begin{equation}
1\leq\frac{u}a+\frac1b+\frac1c\leq1-\frac1c-\frac{e-1}b+\frac1a\leq1-\frac1c+\frac1a.
\end{equation}
In particular, $a<c$.

\medskip
\noindent{\bf Case 1.}

Suppose that $e>1$.
Then $e\geq3$ since $e$ is odd, and by (3) we have
\[
\frac1a%
\ge \frac{e-1}b+\frac1c>\frac1b+\frac1c.
\]
Now by (1) and (2) we have
\[
\gamma>u\alpha=\frac{u}{2a}>\frac12(1-\frac1b-\frac1c)>\frac14.
\]
Hence
\begin{equation}
\psi(2\gamma)=1-2\gamma<\frac1b+\frac1c<\frac1a=2\alpha<\alpha+\beta.
\end{equation}
If $\psi(2\gamma)<\alpha$ define 
$v=\min\{\rho\in\mathbb{Z}:\rho\psi(2\gamma)+\alpha>\beta\}$.
Then $v\geq1$ and $(\alpha,\beta,2v\gamma)$ is good, and
\[
\frac{2v}c\leq{v}(1-2\gamma)=v\psi(2\gamma)<\beta<2\alpha=\frac1a.
\]
Hence
\[
\frac1a+\frac1b+\frac{2v}c<\frac2a+\frac1b<1,
\]
and the result follows with $(r,s,t)=(1,1,2v)$.

If $\alpha<\psi(2\gamma)<\alpha+\beta$ then $(\alpha,\beta,2\gamma)$ is good, and
\[
\frac1a+\frac1b+\frac2c<1,
\]
unless $\{a,b\}=\{3,5\}$ and $c=4$.
The result follows with $(r,s,t)=(1,1,2)$.

\medskip
\noindent{\bf Case 2.}

Suppose that $e=1$.
Then $\beta>\alpha$ implies that $b<a<c$.
In particular, if $(a,b,c)\not=(4,3,5)$ then $a\geq4$, $b\geq3$ and $c\geq7$.
Now
\[
\gamma>u\alpha=\frac{u}{2a}>\frac12(1-\frac1b-\frac1c)>\frac14.
\]
Hence
\[
\psi(2\gamma)=1-2\gamma<\frac1b+\frac1c\leq\frac{10}{21}.
\]
If $\psi(2\gamma)<\alpha$,
define $x=\min\{\sigma\in\mathbb{Z}:\sigma\psi(2\gamma)+\alpha>\beta\}$.
Then $x>0$ and $(\alpha,\beta, 2x\gamma)$ is good, and
\[
\frac{2x}c\leq{x}\psi(2\gamma)<\beta=\frac1{2b}.
\]
From this last inequality it follows that
\[
\frac1a+\frac1b+\frac{2x}c<\frac1a+\frac3{2b}<1,
\]
and the result follows with $(r,s,t)=(1,1,2x)$.

If $\alpha<\psi(2\gamma)<\beta$, 
define $y=\min\{\sigma\in\mathbb{Z}:\psi(2\gamma)+\sigma\alpha>\beta\}$.
Then $y>0$ and $(y\alpha,\beta,2\gamma)$ is good, and
\[
\frac2c\leq\psi(2\gamma)\leq\beta-(y-1)\alpha=\frac1{2b}-\frac{y-1}{2a}.
\]
From this last inequality it follows that
\[
\frac{y}a+\frac1b+\frac2c<\frac{y+1}{2a}+\frac3{2b}<\frac1{2a}+\frac4{2b}<1,
\]
and the result follows with $(r,s,t)=(y,1,2)$.

We are reduced to the case where $\alpha<\beta<\psi(2\gamma)$.
Now define $z=\min\{\sigma\in\mathbb{Z}:\sigma\alpha+\beta>\psi(2\gamma)\}$.
Then $z>0$ and $(z\alpha,\beta,2\gamma)$ is good.

Moreover, adding the inequalities
\[
(u-1)\alpha+\beta\leq\gamma
\]
and
\[
(z-1)\alpha+\beta\leq1-2\gamma=\psi(2\gamma)
\]
gives
\begin{equation}
(2u+z-3)\alpha+3\beta\geq\frac{(2+z-3)}a+\frac12\leq1.
\end{equation}
Combining (2) and (5) gives
\[
\frac{z}a+\frac1b+\frac2c\leq\frac12-\frac{2u-3}a+\frac1b+\frac2c\leq
\frac3a+\frac3b+\frac4c-\frac32\leq\frac34+\frac33+\frac47-\frac32<1.
\]
The result follows with $(r,s,t)=(z,1,2)$.
This completes the proof.
\end{proof}

Simple calculations show that the conclusion of Lemma 2 does not hold for
$(a,b,c)=(3,4,5)$ and $(d,e,f)=(1,1,3)$ or $(1,3,1)$.

The conclusions of Lemma 2 are enough to prove that certain groups have weight strictly greater than 1.

\begin{lemma}
Let $a,b,c$ be distinct quasi-primes, with $\{a,b,c\}\not=\{3,4,5\}$.
Then the group $G$ with presentation
$\langle{u,x,y,z}\mid{u^2=xyz},~x^a=y^b=z^c=1\rangle$ has weight $2$.
\end{lemma}

\begin{proof}
The group $G$ is clearly the normal closure of $\{xy,u\}$, and so has weight at most 2.
It shall suffice to show that $G$ cannot have weight 1.
Let $w$ be a word in $u,x,y,z$, and let $E_v(w)$ be the exponent sum of $v\in\{u,x,y,z\}$ in $w$.
Let $d=E_u(w)+2E_x(w)$, $e=E_u(w)+2E_y(w)$ and $f=E_u(w)+2E_z(w)$.

If $a$ divides $d$ then $w$ is contained in the normal closure of $\{y,z\}$ in $G$,
and so cannot be a weight element.
Analogous remarks hold if $b$ divides $e$ or $c$ divides $f$.
Similarly, if $E_u(w)$ is even then $w$ is in the normal closure of $\{x,y,z\}$,
and so cannot be a weight element.
We may assume that none of these arithmetic conditions holds for $w$. 
Then $(d,2a)=(e,2b)=(f,2c)=1$.
Hence there are integers $r,s,t$ which are coprime to $a,b,c$,
respectively, and such that
$\frac{r}a+\frac{s}b+\frac{t}c<1$ and $(\frac{rd}{2a},\frac{se}{2b},\frac{tf}{2c})$ is good,
by Lemma 2.

We shall show that $w$ is not a weight element by proving the existence of a non-trivial representation $\rho:G\to{SU(2)}$ with $\rho(w)\in\{\pm{I}\}$.
Let $X,Y$ and $Z\in{SU(2)}$ be (variable) matrices with trace $2\cos(\frac{r\pi}a)$,
$2\cos(\frac{s\pi}b)$ and $2\cos(\frac{t\pi}c)$, respectively.
Thus each of $X,Y$ and $Z$ varies in a space homeomorphic to $S^2$.
Moreover, it follows from the condition $\frac{r}a+\frac{s}b+\frac{t}c<1$ that,
for any choice of $X,Y,Z$, the trace of $XYZ$ is strictly greater than $-2$, 
and hence $XYZ\not=-I$.
Every matrix $M\in{SU(2)}\setminus\{-I\}$ has an unique square root $U$ with positive trace.
Moreover, $U$ varies continuously with $M$.
Thus, taking $M=XYZ$, we obtain an uniquely defined $U$, continuous in $X,Y,Z$, such that $U^2=XYZ$ and $I=\pm{X^a}=\pm{Y^b}=\pm{Z^c}$.
Thus we have a family of representations $\rho_{\ell,m,n}:G\to{SO(3)}=SU(2)/\{\pm{I}\}$,
which is continuously parametrized by $(\ell,m,n)\in{S^2}\times{S^2}\times{S^2}$.

If $X,Y,Z$ are chosen to be mutually commuting matrices (in other words, if $\ell=\pm{m}=\pm{n}$
in $S^2$) then the trace of $\rho_{\ell,m,n}(w)$ is $2\cos\theta$,
where 
\[
\theta=\pm\frac{rd}{2a}\pm\frac{se}{2b}\pm\frac{tf}{2c}.
\]
If $w$ is a weight element for $G$ then $\rho_{\ell,m,n}(w)\not=\pm{I}$ for any triple $(\ell,m,n)$.
In particular, $\rho_{\ell,\pm\ell,\pm\ell}(w)$ define $S^1$-equivariant maps from $S^2$ to 
$SU(2)\setminus\{\pm{I}\}$.
Since the triple $(\frac{rd}{2a},\frac{se}{2b},\frac{tf}{2c})$ is good,
it follows from Lemma 1 and Corollary 2.2 of \cite{Ho} that these maps belong to distinct homotopy classes. More specifically, fix $\ell=i$, corresponding to the diagonal matrix
\[
D(i)=\left(\begin{smallmatrix} i&0\\0&-i\end{smallmatrix}\right)\in{SU(2)}.
\]
Then we have 
\[\rho (\ell, \ell, \ell)=D(exp(i\pi\alpha)),
\quad\quad
\rho (\ell, \ell, -\ell)=D(exp(i\pi\beta)),
\]
\[
\rho (\ell, -\ell, \ell)=D(exp(i\pi\gamma))\quad\mathrm{
 and}\quad\rho (\ell, -\ell, -\ell)=D(exp(i\pi\delta)),
 \]
 where
 $\alpha=\frac{rd}{2a}+\frac{se}{2b}+\frac{tf}{2c}$,
 $\beta=\frac{rd}{2a}+\frac{se}{2b}-\frac{tf}{2c}$,
 $\gamma=\frac{rd}{2a}-\frac{se}{2b}+\frac{tf}{2c}$
 and $\delta=\frac{rd}{2a}-\frac{se}{2b}-\frac{tf}{2c}$.
The homotopy class of $\rho (\ell, \ell, -\ell)$ (resp., $\rho (\ell, -\ell, -\ell)$)
in $\pi_2(SU(2)\setminus\{\pm{I}\})$ is determined by the signs of the imaginary parts of
$\alpha,\beta$ (resp., $\gamma,\delta$), by Corollary 2.2 of \cite{Ho}.
Since $(\frac{rd}{2a},\frac{se}{2b},\frac{tf}{2c})$ is good, 
Lemma 1 and this Corollary together imply that the maps
$\rho (\ell, \ell, -\ell)$ and $\rho (\ell, -\ell, -\ell)$ belong to different homotopy classes.

Since the parameters $(\ell,\pm\ell)$ belong to a connected parameter space $S^2\times{S^2}$, this is a contradiction (as in \cite{Ho}).
Thus our assumption that the normal closure of $w$ is $G$ must be false, as required.
\end{proof}

Let $G$ be a group with a presentation $\mathcal{P}$ 
with $g$ generators and $r$ relators, 
and let $m(\mathcal{P})$ be the $r\times{g}$ matrix 
with $(i,j)$ entry the exponent sum of the $j$th generator in the $i$th relator.
Then $G$ has abelianization $G/G'\cong\mathbb{Z}$ 
if and only if $m(\mathcal{P})$ has rank $g-1$
(so all $g\times{g}$ minors of $m(\mathcal{P})$ are 0) 
and the highest common factor of the $(g-1)\times(g-1)$ minors 
is $1$.

\begin{theorem}
Let $B$ be an aspherical $2$-orbifold $B$ such that $\beta=\pi^{orb}(B)$ has weight $1$.
Then $B$ is either
 \begin{enumerate}
\item$S^2(a_1,\dots,a_m)$, 
with $m\geq3$, no three of the cone point orders $a_i$ 
have a nontrivial common factor,
and at most two disjoint pairs each have a common factor;
\item$P^2(b_1,\dots,b_m)$ 
with $m=2$ or $3$,
the cone point orders $b_i$ being pairwise relatively prime, and $b_1=2$ if $m=3$;
or 
\item $P^2(3,4,5)$; or
\item$\mathbb{D}(c_1\dots,c_p,\overline{d_1},\dots,\overline{d_q})$,
with $p\leq2$ and $2p+q\geq3$,
the cone point orders $c_i$ being all odd and relatively prime,
and at most one of the $d_j$ being even.
\end{enumerate} 
\end{theorem}

\begin{proof}
Since $\beta$ has cyclic abelianization,
the surface underlying $B$ is $S^2$, $P^2$ or $D^2$.

The groups $\beta$ for such orbifolds have presentations
\[
\langle{v_1,\dots,v_m}\mid{v_i^{a_i}=1}~\forall{i}\leq{m},
~\Pi{v_i}=1
\rangle,
\]
\[
\langle{u,v_1,\dots,v_m}\mid{u^2=v_1\cdots {v_m,}~v_i^{b_i}=1~\forall{i}\leq{m}}\rangle,
\quad\mathrm{and}
\]
\[
\langle{v_1,\dots,v_p,x_1,\dots,x_{q+1}}\mid 
v_i^{c_i}=1~\forall~i\leq{p},~x_j^2=(x_jx_{j+1})^{d_j}=1~\forall~j\leq{q},
\]
\[x_{q+1}\Pi{v_i}=(\Pi{v_i})x_1
\rangle,
\]
respectively.
Here the $v_i$ are orientation preserving, 
while $u$ and the $x_j$ are orientating reversing.

Most of the details on the parity and common factors of the cone point orders 
follow from the fact that $\beta/\beta'$ is cyclic
and the above observations on determinants,
while the lower bounds on the numbers of cone points 
and corner points hold because $B$ is aspherical.

In case (1), if $p$ disjoint pairs each have proper common factors
then we may assume that $\alpha_1=(a_1,a_2),\dots,\alpha_p=(a_{2p-1},a_{2p})$ 
are all greater than 1.
Let $\nu$ be the normal subgroup generated by $\{v_1v_2,\dots,v_{2p-1}v_{2p}\}$
and $\{v_i:p<i\leq{m}\}$.
Then $\beta/\nu\cong\ast_{i=1}^p{Z/\alpha_iZ}$.
Since it has weight 1, $p\leq2$ \cite{Ho}.

In case (2), if $m>3$ or if $m=3$ and $b_1>2$ but $(b_1,b_2,b_3)\not=(3,4,5)$
then $\beta$ has a quotient $G$ as in Lemma 3, and so cannot have weight 1.
Thus (2) or (3) must hold.

In case (4), the free product $\ast_{i=1}^p{Z/c_iZ}$ 
is a quotient of $\beta$.
Since it has weight one, $p\leq2$ \cite{Ho}.
\end{proof}

The two natural extensions of the Scott-Wiegold conjecture
which are raised in the final section of \cite{Ho} 
each suggest strong bounds on $m$ in case (1).
Firstly, does every free product of three nontrivial groups have weight $>1$?
Secondly, does every free product of $2k+1$ finite cyclic groups have weight $>k$?

If $m\geq9$ and $\nu$ is the normal subgroup generated by $\{v_1v_2v_3,v_4v_5v_6\}$
then $\beta/\nu$ is the free product of three nontrivial groups.
Thus if such free products always have weight $>1$ we must have $m\leq8$.

If every free product of 5 or more finite cyclic groups has weight $>1$
then $m\leq4$ in case (1).
When $m=3$ every such group has weight 1,
for we may assume that $a_1$ is odd,
and then $v_1^{-1}v_2$ is a normal generator.
If $m=4$ and $(a_1,a_2)=(a_3,a_4)=1$ then $v_1v_2$ is a normal generator.
However, if $m=4$ and the exponents  do not form two relatively prime pairs 
then the group has a quotient with presentation 
\[
\langle{x,y,z}\mid {x^a=y^{bc}=z^{bd}=(xyz)^{cd}=1}\rangle,
\]
where  $a,b,c$ and $d$ are distinct primes.
Do such groups have weight 1?

In case (2), if $m=2$ then $v_1^{-1}u$ is a normal generator.
If there are any examples with base $P(c_1,\dots,c_m)$ and $m=3$ 
then
\[
\langle{u,x,y,z}\mid {u^2=xyz,~x^2=y^b=z^c=1}\rangle
\]
has weight 1 for some distinct odd primes $b<c$.
(This applies also to case (3).)
This seems unlikely.

If $B=\mathbb{D}(\overline{d_1,\dots,d_q})$
or $B=\mathbb{D}(c,\overline{d_1,\dots,d_q})$,
where the $d_i$ with $i\geq2$ are all odd,
then $\beta$ has weight 1 for any $q>0$.
If there are any examples with $p=2$ then
\[
\langle{v,w,x}\mid{v^a=w^b=x^2=1,~xvw=vwx}\rangle
\]
has weight 1 for some distinct odd primes $a<b$.
Again, this seems unlikely.
In summary, we expect only $m=3$ or 4 in case (1), 
$m=2$ in case (2) and $p\leq1$ in case (4).

We shall give more details on the low dimensional cases $n=1,2$ or 3 in subsequent sections.

\section{the classical case}

Let $M(0;S)$ be the Seifert fibred 3-manifold with base orbifold
$B=S^2(\alpha_1,\dots\alpha_r)$ and Seifert data $S=\{(\alpha_1,\beta_1),\dots,(\alpha_r,\beta_r)\}$.
(Here $0$ is the genus of the surface underlying the base orbifold,
and the generalized Euler invariant is $\varepsilon=-\Sigma\frac{\beta_i}{\alpha_i}$.)
Then the knot manifold of the $(p,q)$-torus knot $k_{p,q}$ 
is $M(0;S)$, where $S=\{(p,q),(q,p),(pq,-p^2-q^2)\}$.
(See Lemma 7 below.)
This knot is fibred, and it has Alexander polynomial 
$\Delta_1(k_{p,q})=\frac{(t^{pq}-1)(t-1)}{(t^p-1)(t^q-1)}$,
which is a square-free product of cyclotomic polynomials.
We shall extend these properties to other knots whose associated knot manifolds 
are Seifert fibred. 

\begin{theorem}
If $K$ is a nontrivial knot such that $M(K)$ is Seifert fibred,
with Seifert fibration $p:M\to{B}$, then 
\begin{enumerate}

\item $B=S^2(a_1,\dots,a_m)$, with $m\geq3$, and $\varepsilon(p)=0$; 

\item {no} cone point order $a_i$ is relatively prime to all the others,
at most two disjoint pairs each have a common factor,
and no three have a nontrivial common factor;

\item$\Sigma\frac1{a_i}\leq{m-2}$;

\item $K$ is fibred; and

\item $\Delta_1(K)/\Delta_2(K)$ is a square-free product of cyclotomic polynomials.
\end{enumerate}
\end{theorem}

\begin{proof}
Since $K$ is nontrivial, $M=M(K)$ is aspherical \cite{Ga}.
Let $p:M\to{B}$ be the projection of the Seifert fibration, 
let $h$ be the image of the regular fibre in $\pi=\pi_1(M)$,
and let  $\beta=\pi^{orb}(B)$.
Since $\beta/\beta'$ is cyclic it is finite, and so the image of $h$ in
$\pi/\pi'=H_1(M)$ has infinite order.
Therefore $\beta$ acts trivially on $h$,
and so $B$ is orientable, since $M$ is orientable.
Hence $B=S^2(a_1,\dots,a_m)$, with $m\geq3$,
by Theorem 4.
This theorem also implies that no three of the  $a_i$ have a nontrivial common factor,
and at most two disjoint pairs each have a common factor.

The subgroup $\langle\pi',h\rangle\cong\pi'\times\mathbb{Z}$ has finite index
in $\pi$, and so $\pi'$ is finitely presentable.
Therefore $M$ fibres over $S^1$, and the monodromy has finite order in $Out(\pi')$.
Moreover, $\varepsilon(p)$ is 0,
and so no $a_i$ is relatively prime to all the others.

If $K=k_{2,3}$ is the trefoil knot then $\Sigma\frac1{a_i}=1$.
In all other cases $B$ must be a hyperbolic 2-orbifold, 
and so $\chi^{orb}(B)=2-m+\Sigma\frac1{a_i}<0$.

Let $\lambda$ be a longitude for $K$.
Then $\pi\cong\pi{K}/\langle\langle\lambda\rangle\rangle$,
and so  $\pi/\pi''\cong\pi{K}/\pi{K}''$, since $\lambda\in\pi{K}''$.
A choice of meridian for $K$ determines an isomorphism
$\mathbb{Z}[\pi/\pi']\cong\Lambda=\mathbb{Z}[t,t^{-1}]$,
and the annihilator of $\pi'/\pi''$ as a $\Lambda$-module
is generated by $\Delta_1(K)/\Delta_2(K)$ \cite{Cr}.
Since $M$ is fibred, so is $K$ \cite{Ga},
and since the monodromy of the fibration of $M$ has finite order,
$\Delta_1(K)/\Delta_2(K)$ is a square-free product of cyclotomic factors.
\end{proof}

The various arithmetic conditions listed above are not in themselves sufficient
to impose any bound on $m$.
The condition $\Sigma\frac1{a_i}\leq{m-2}$ is only of interest when $m=3$.

If $K=k_{2,3}$ is the trefoil knot then $M(K)$ 
is the flat 3-manifold $G_5$ with holonomy $Z/6Z$,
which is Seifert fibred over $S^2(2,3,6)$.

\begin{cor}
If $K$ is not the trefoil knot then $M(K)$
is a $\mathbb{H}^2\times\mathbb{E}^1$-manifold.
\end{cor}

\begin{proof}
If $K$ satisfies the above conditions but is not the trefoil knot
then it has genus $\geq2$, 
since the figure-eight knot is the only other fibred knot
of genus 1.
Hence $M$ is an $\mathbb{H}^2\times\mathbb{E}^1$-manifold,
since  $\varepsilon(p)=0$.
\end{proof}

Examples with $m=3$ or 4 are easy to find.

\begin{lemma}
If  $p>q>(p,q)=1$ then $M(k_{p,q})$ and $M(k_{p,q}\#-k_{p,q})$ are Seifert fibred, 
with three and four exceptional fibres, respectively.
\end{lemma}

\begin{proof}
The Seifert fibration of the exterior $X(k_{p,q})$ has two exceptional fibres, 
of multiplicities $p$ and $q$.
This fibration extends to $M(k_{p,q})=X(k_{p,q})\cup{D^2\times{S^1}}$,
with the core of the solid torus having multiplicity $pq$,
and so $M(k_{p,q})$ is Seifert fibred over $S^2(p,q,pq)$.

In general, $M(K\#-L)=X(K)\cup{-X(L)}$, where the boundaries are identified using 
the canonical meridian-longitude coordinizations. 
If $K$ and $L$ are torus knots then the Seifert fibrations on the boundaries 
agree if and only if $K=L$.
Hence $M(k_{p,q}\#-k_{p,q})$ 
is Seifert fibred over $S^2(p,p,q,q)$.
\end{proof} 

It is much harder to find other examples.
(If  Conjecture (2) of \cite{Ho} is resolved as expected then 
each Seifert fibred knot manifold must have at most 4 exceptional fibres.)
The first knot in the standard tables which satisfies 
the conditions of Theorem 5, but which is not a torus knot, 
is $10_{132}$.
(This is also known as the Montesinos knot 
$\mathfrak{m}(-1;(2,1),(3,1),(7,2))$,
in the notation of \cite{BZ}.)
Its knot manifold $M(10_{132})$ is Seifert fibred, 
with base $S^2(2,3,19)$.
This is the only hyperbolic arborescent knot whose 
associated knot manifold is a Seifert manifold with
three exceptional fibres \cite{Me,Wu}.
If $K$ is an alternating knot such that $M(K)$ is Seifert fibred 
must $K$ be a $(2,q)$-torus knot?

If $K$ is any other knot such that $M(K)$ is Seifert fibred with 
more than three exceptional fibres
then $K$ is either hyperbolic or is a satellite of a torus knot, 
with the knot exterior having a JSJ decomposition into the union of
a torus knot exterior and a hyperbolic piece \cite{MM}.

\section{aspherical seifert fibred $2$-knot manifolds}

A 2-knot is a locally flat embedding of $S^2$ in a homotopy 4-sphere.
The constructions that we shall use give PL 2-knots in PL homotopy 4-spheres.
On the other hand, all homotopy 4-spheres are homeomorphic to $S^4$.
The groups of 2-knots with aspherical, Seifert fibred knot manifolds
may be characterized as follows.

\begin{theorem}
A group $\pi$ is the group of a $2$-knot $K$ such that $M(K)$ is an
aspherical, Seifert fibred $4$-manifold if and only if $\pi$ is an orientable
$PD_4$-group of weight $1$ and with a normal subgroup $A\cong\mathbb{Z}^2$.
\end{theorem}

\begin{proof}
The conditions are obviously necessary.
Suppose that they hold.
Then $\chi(\pi)=0$, and so $\beta_1(\pi)=\frac12\beta_2(\pi)+1>0$.
Since $H_1(\pi;\mathbb{Z})$ is cyclic, 
we must have $H_1(\pi;\mathbb{Z})\cong\mathbb{Z}$,
and then $H_2(\pi;\mathbb{Z})=0$, 
by Poincar\'e duality and the Universal Coefficient Theorem.
Thus $\pi$ satisfies the Kervaire conditions.
It follows from the cohomology Lyndon-Hochschild-Serre spectral sequence for $\pi$ 
as an extension of $\pi/A$ by $A$ that $H^2(\pi/A;\mathbb{Z}[\pi/A])\cong\mathbb{Z}$.
Therefore $\pi/A$ is virtually a $PD_2$-group \cite{Bo},
and so $\pi=\pi_1(M)$, where $M$ is a Seifert fibred 4-manifold.
Surgery on a simple closed curve in $M$ representing a normal generator of $\pi$
gives a homotopy 4-sphere.
The cocore of the surgery is a 2-knot $K$ with $M(K)\cong{M}$ and $\pi{K}\cong\pi$.
\end{proof}

The monodromy of a Seifert fibration $p:M\to{B}$ with total space a 4-manifold,
base an aspherical 2-orbifold and general fibre a torus,
is {\it diagonalizable\/} if it is generated by a matrix which
is conjugate to a diagonal matrix in $GL(2,\mathbb{Z})$.
The possible base orbifolds and monodromy actions are largely determined in the next result.
(In Theorem 16.2 of \cite{Hi} it is shown that the knot manifolds are $s$-cobordant to
geometric 4-manifolds.)

\begin{theorem}
Let $K$ be a $2$-knot with group $\pi=\pi{K}$,
and such that the knot manifold
$M=M(K)$ is Seifert fibred, with base $B$.
If $\pi'$ is infinite then $M$ is aspherical and $B$ is either
 \begin{enumerate}
\item$S^2(a_1,\dots,a_m)$, with $m\geq3$, 
no three of the cone point orders $a_i$ have a nontrivial common factor, 
at most two disjoint pairs each have a common factor, and trivial monodromy; or
\item$P^2(b_1,\dots,b_m)$ with $m=2$ or $3$,
the cone point orders $b_i$ being pairwise relatively prime, and $b_1=2$ if $m=3$,
or  $P^2(3,4,5)$,
and monodromy of order $2$ and non-diagonalizable; or
\item$\mathbb{D}(c_1\dots,c_p,\overline{d_1},\dots,\overline{d_q})$,
with the cone point orders $c_i$ all odd and relatively prime,
and at most one of the $d_j$ even,  $p\leq2$ and $2p+q\geq3$, 
and monodromy of order $2$ and diagonalizable.
\end{enumerate}
All such knot manifolds are mapping tori,
and are geometric.
\end{theorem}

\begin{proof}
Let $A$ be the image of the fundamental group of the regular fibre in $\pi$.
Then $A$ is a finitely generated infinite abelian normal subgroup of $\pi$,
and $\beta=\pi^{orb}(B)\cong\pi/A$.
If $M$ were not aspherical then $\beta$ would be finite,
so $\pi$ would be virtually abelian. 
But then $\pi'$ would be finite, by Theorem 15.14 of \cite{Hi}.
Therefore $M$ is aspherical, so $A\cong\mathbb{Z}^2$
and $B$ is as in Theorem 4.
Hence $A\cap\pi'\cong\mathbb{Z}$ and $\beta/\beta'$ is finite, 
by Theorem 16.2 of \cite{Hi},
and so $\beta'$ and $\pi'$ are finitely presentable.
Therefore $\pi$ is virtually $\pi'\times\mathbb{Z}$.

Let $\Theta:\beta\to{GL}(2,\mathbb{Z})$
be the action of $\beta$ on $A$ induced by conjugation in $\pi$.
Since $M$ is orientable,
$\det\Theta(g)=-1$ if and only if $g\in\beta$ is orientation-reversing.
Since $A\cap\pi'\cong\mathbb{Z}$,
every element of $\beta$ has at least one eigenvalue $+1$.
Since $\beta/\beta'$ is finite cyclic, 
orientation-preserving elements of $\beta$ act trivially on $A$.

If $B=S^2(a_1,\dots,a_m)$ then $A$ is central in $\pi$,
and $\pi$ has the presentation
\[
\langle{x_1,\dots,x_m,y,z}\mid{x_i^{a_i}=y^{e_i}z^{f_i}}~{and}
~x_i\leftrightharpoons{y,z}~\forall{i}\leq{m},
~\Pi{x_i}=y^kz^l,\]
\[yz=zy\rangle,
\]
for some exponents $e_i,f_i$, for $1\leq{i}\leq{m}$, and $k,l$.
Conversely, if $\pi$ has such a presentation 
then $\pi/\pi'\cong\mathbb{Z}$ if and only if the $m+2$ minors 
\[
(k-\Sigma\frac{e_i}{a_i})\Pi{a_i},
~(l-\Sigma\frac{f_i}{a_i})\Pi{a_i},~\mathrm{ and}~
(kf_1-le_1+\Sigma_{j\not=1}\frac{e_1f_j-e_jf_1}{a_j})\Pi{a_i},
\]
\[
\dots,~
(kf_m-le_m+\Sigma_{j\not=m}\frac{e_mf_j-e_jf_m}{a_j})\Pi{a_i}
\]
have highest common factor 1.
If so, then  $(a_i,e_i,f_i)=1$ for all $i$,
and so $\pi$ is torsion free.

If $B=P^2(b_1,\dots,b_m)$, for some $m\geq2$,
then $\beta$ has the presentation
\[
\langle{u,x_1,\dots,x_m}\mid{u^2=x_1\cdots {x_m,}~x_i^{b_i}=1~\forall{i}\leq{m}}\rangle,
\]
where $u$ is the only orientation-reversing generator.
Since $\Theta(x_i)=I$ for $1\leq{i}\leq{m}$,
the final relation implies that $\Theta(u)^2=I$.
We may choose the basis for $A$ so that $\Theta(u)$ 
has one of the standard forms
$\left(\begin{smallmatrix}
1&0\\ 0&-1
\end{smallmatrix}\right)$
or  $\left(\begin{smallmatrix}
0&1\\ 1&0
\end{smallmatrix}\right)$. 
If $\Theta(u)$ is diagonalizable then $\pi$ has the presentation
\[
\langle{u,x_1,\dots,x_n,y,z}\mid{u^2=x_1\cdots x_my^kz^l,}~x_i^{b_i}=y^{e_i}z^{f_i}~
~\forall{i}\leq{m},
\]
\[
x_i\leftrightharpoons{y,z}~\forall{i}\leq{m},~uy=yu,~uzu^{-1}=z^{-1},~yz=zy\rangle,
\]
for some exponents $e_i,f_i$, for $1\leq{i}\leq{m}$, and $k,l$.
But then $\pi/\pi'$ has 2-torsion.
Therefore $\Theta(u)=
\left(\begin{smallmatrix}
0&1\\ 1&0
\end{smallmatrix}\right)$ and $\pi$ has the presentation
\[
\langle{u,x_1,\dots,x_m,y,z}\mid{u^2=x_1\cdots x_my^kz^l,}
~x_i^{b_i}=y^{e_i}z^{f_i}~~\forall{i}\leq{m},
\]
\[
x_i\leftrightharpoons{y,z}~\forall{i}\leq{m},~uyu^{-1}=z,~yz=zy\rangle.
\]
Such a group has abelianization $\pi/\pi'\cong\mathbb{Z}$ if and only if $(b_i,e_i+f_i)=1$, for $1\leq{i}\leq{m}$,
and either one exponent $b_i$ is even or $k+l+\Sigma(e_i+f_i)$ is odd.
If so, then $\pi$ is torsion free.

Suppose finally that $B=\mathbb{D}(c_1,\dots,c_p,\overline{d_1},\dots,
\overline{d_q})$, for some $p,q\geq0$ with $2p+q\geq3$. 
Then $\beta$ has the presentation
\[
\langle{w_1,\dots,w_p,x_1,\dots,x_{q+1}}\mid 
w_i^{c_i}=1~\forall~i\leq{p},~x_j^2=(x_jx_{j+1})^{d_j}=1~\forall~j\leq{q},
\]
\[x_{q+1}\Pi{w_i}=(\Pi{w_i})x_1
\rangle,
\]
where the generators $x_j$ are orientation-reversing.
Since the products $x_jx_{j+1}$ are orientation preserving and of finite order,
$\Theta(x_j)=\Theta(x_{j+1})$ for all $j\leq{q}$.
Since the subgroups generated by $x_j$ and $A$ are torsion-free,
we may choose the basis for $A$ so that 
$\Theta(x_j)=\left(\begin{smallmatrix}
1&0\\ 0&-1
\end{smallmatrix}\right)$ for all $j$.
Hence $\pi$ has the presentation
\[
\langle{w_1,\dots,w_p,x_1,\dots,x_{q+1},y,z}\mid 
{w_i^{c_i}=y^{e_i}z^{f_i}}~and~w_i\leftrightharpoons{y,z}~\forall{i}\leq{p},
\]
\[x_j^2=y~\forall~j\leq{q+1},~x_jzx_j^{-1}=z^{-1}~and~(x_jx_{j+1})^{d_j}
=y^{g_j}z^{h_j}~\forall~j\leq{q},
\]
\[
x_{q+1}\Pi{w_i}=(\Pi{w_i})x_1y^kz^l\rangle,
\]
for some exponents $e_i,\dots,h_j,k,l$.
Since $\pi$ is torsion-free, $(c_i,e_i,f_i)=1$ for all $i\leq{p}$.
Since $x_{q+1}^2=x_1^2$, $k=0$.
Extending coefficients to $\mathbb{Z}[\frac12]$ 
(localizing away from $(2)$),
we see that $\pi/\pi'$ is infinite cyclic 
if and only if $g_j=d_j$ for all $j\leq{q}$,
and then $\pi/\pi'$ has rank 1 and no odd torsion.
Reducing modulo $(2)$, we then see that $\pi/\pi'\cong\mathbb{Z}$ 
if and only if $e_j=d_j$ and $(2,d_j,h_j)=1$ for all $j\leq{q}$.

Since the monodromy is finite, these Seifert fibred manifolds are geometric, with geometry $\mathbb{N}il^3\times\mathbb{E}^1$ if the base is flat and geometry
$\widetilde{\mathbb{SL}}\times\mathbb{E}^1$ if the base is hyperbolic
\cite{Ue}.
Hence $M(K)$ fibres over $S^1$.
\end{proof}

Every torsion-free extension of $\beta$ by $\mathbb{Z}^2$ as in Theorem 9 
and with abelianization $\mathbb{Z}$ is the fundamental group of 
a Seifert fibred homology $S^1\times{S^3}$, 
which may be obtained by surgery on a knot in an homology 4-sphere.
Such an extension has weight 1 if and only if $\beta$ has weight 1.

If $K$ is the $r$-twist spin of a classical knot then the $r$th power of a meridian 
is central in $\pi{K}$, and so must have trivial image in $\beta$.
Since elements of finite order in 2-orbifold groups are conjugate to
cone point or reflector generators (see Theorem 4.8.1 of \cite{ZVC}), 
the only possible bases for Seifert fibred twist spins are
$S^2(a,b,r)$, with $(a,b)=1$, and $\mathbb{D}(\overline{d_1,\dots,d_q})$.
In the latter cases, $r$ must be 2.
These are realized by $r$-twist spins of $(a,b)$-torus knots and 2-twist spins 
of Montesinos knots $\mathfrak{m}(e;(d_1,\beta_1),\dots,(d_q,\beta_q))$,
respectively.

When $B=P^2(b_1,\dots,b_m)$ the knot manifold is also the total space 
of an $S^1$-bundle over a Seifert fibred homology $S^1\times{S^2}$,
since $\pi$ has an infinite cyclic normal subgroup 
$\langle{z}\rangle$ with torsion free quotient.
In the other cases whether this is so depends on 
the exponents $e_i,\dots,k,l$ in the relators.

There are no known examples of 2-knot manifolds which are total spaces
of orbifold bundles with flat bases and hyperbolic general fibre.
See \cite{Hi13} for a discussion of the possibilities.

\section{flat bases}

The three possible flat bases for Seifert fibred knot manifolds are
$S^2(2,3,6)$, $\mathbb{D}(\overline{3,3,3})$.
and $\mathbb{D}(3,\overline{3})$.
According to Theorem 6.11 of \cite{Hi89}, 
there should be no examples with base $\mathbb{D}(3,\overline{3})$.
However, in seeking to understand why that should be so,
we have recently realised that there was an error in the claim 
made there that certain $\mathbb{N}il^3$-lattices 
have essentially unique meridianal automorphisms.

Let $G=\pi_1(M(0;(3,1),(3,1),(3,1-3e))$, where $e$ is even.
Then $G$ has the presentation
\[
\langle{h,x,y,z}\mid{x^3=y^3=z^3=h,~xyz=h^e}\rangle
\]
\[
=\langle{x,z}\mid{x^3=(x^{3e-1}z^{-1})^3=z^3}\rangle.
\]
The centre of $G$ is $\zeta{G}=\langle{h}\rangle$, 
and $\overline{G}=G/\zeta{G}$ is the flat 2-orbifold group $\pi_1^{orb}(S^2(3,3,3))\cong\mathbb{Z}^2\rtimes{Z/3Z}$,
with translation subgroup $T\cong\mathbb{Z}^2$
generated by the images of $u=z^{-1}x$ and $v=xz^{-1}$.

The group of outer automorphism classes  $Out(G)$ is generated by 
the images of the automorphisms $b,r$ and $k$,
where
\[
b(x)=z^2x^{3e-4},\quad{b(z)=z},\quad{r(x)=x^{-1}},\quad{r(z)=z^{-1}},
\]
\[
k(x)=xz^{-1}x^{3e-2}\quad{and}\quad{k(z)=x}.
\]
In \cite{Hi89} the involution $r$ was called $j$, 
and it was asserted there that $jkj^{-1}=k^{-1}$.
This passed unchanged in \cite{Hi} (Theorem 16.15) and in \cite{Hi11}.
However, it should be $jk=kj$ (i.e., $rk=kr$),
and 
there are two classes of meridianal automorphisms of $G$,
up to conjugacy and inversion in $Out(G)$,
represented by $r$ and $rk$.
(This error did not affect the cases with parameter $\eta=-1$ 
in the earlier work.
In those cases the outer automorphism group is $(Z/2Z)^2$ 
and the meridianal class is unique.)

The automorphism $r$ leads to the group $\pi(e,1)$, considered
in \cite{Hi89,Hi,Hi11}.

\begin{theorem}
Let $e$ be even and let $\pi=G\rtimes_{rk}\mathbb{Z}$ 
be the group with presentation
\[
\langle{t,x,z}\mid{x^3}=(x^{3e-1}z^{-1})^3=z^3,~txt^{-1}=x^{-1}zx^{2-3e},~
tzt^{-1}=x^{-1}\rangle.
\]
Then
\begin{enumerate}
\item$\pi/\pi'\cong\mathbb{Z}$ and $\pi$ has weight $1$;
\item{the} centre of $\pi$ is generated by the image of $(t^3x)^2$;
\item{no} automorphism of $\pi$ sends $t$ to a conjugate of $t^{-1}$;
\item{all} automorphisms of $\pi$ are orientation preserving; and
\item{the} strict weight orbits for $\pi$ are parametrized by $\mathbb{Z}$.
\end{enumerate}
\end{theorem}

\begin{proof}
The image of $t$ freely generates the abelianization,
which is $\mathbb{Z}$ since $e$ is even.
Since $\pi$ is solvable it follows that $t$ normally generates $\pi$.
The second assertion is easily verified by direct computation.
Since $\theta=rk$ is not conjugate to $\theta^{-1}=rk^{-1}$ in $Out(G)$,
no automorphism of $\pi$ inverts $t$.
Since automorphisms of $\mathbb{N}il^3$-lattices are orientation preserving, 
it then follows that all automorphisms of $\pi$ are orientation preserving.

If $\tilde{t}$ is another normal generator  
with the same image in $\pi/\pi'$ as $t$ then
we may assume that $\tilde{t}=tg$ for some $g\in\pi''=G'$, 
by Theorem 14.1 of \cite{Hi}.
Since $\zeta{G}=\langle{x^3}\rangle$ and
$\pi$ has an automorphism $f$ such that $f(g)=g$ for $g\in{G}$
and $f(t)=tx^3$, we may work modulo $\zeta{G}$.
Let $\overline\pi=\pi/\zeta{G}$.
Then $\overline\pi'=\overline{G}\cong\pi_1^{orb}(S^2(3,3,3))$, 
and $T=\sqrt{G}/\zeta{G}=\sqrt\pi/\zeta\pi'$.
Since $\theta$ differs from the meridianal automorphism 
of  $\pi(e,1)'$ only by an automorphism 
which induces the identity on the subquotient $T$, 
the calculation is then as in \cite{Hi11}, for such groups.
However we shall give more details here.

Suppose that $tg$ and $tgh$ are two normal generators,
with $g,h\in\overline{\pi}''$. 
If there is an automorphism $\alpha$ of $\pi$ such that $\alpha(tg)=tgh$
then $\theta{c_{gh}}\gamma=\gamma\theta{c_g}$, 
where $c_g$ is conjugation by $g$ (etc.) and $\gamma=\alpha|_G$.
Hence the images of $\gamma$ and $\theta$ in $Out(\overline{G})$ commute,
and so $\gamma=c_j\delta$ for some automorphism 
$\delta\in\langle{r,k}\rangle$ and some $j\in\overline{G}$.
Since $\delta\theta=\theta\delta$ in $Aut(\overline{G})$, 
we have $c_{ghj}=\theta^{-1}c_j\theta\delta{c_g}\delta^{-1}$,
i.e., $ghj=\theta^{-1}(j)\delta(g)$.
Now $\theta$ induces inversion on $G/\sqrt\pi=Z/3Z$.
Hence $j$ must be in $\sqrt{\overline{\pi}}$,
and so $h=(\Theta^{-1}-I)([j])+(\Delta-I)(g)$
where $\Theta$ and $\Delta$ are the automorphisms of $T$
induced by $\theta$ and $\delta$.
Since $\Delta$ is a power of $\Theta$,
it follows that $h\in\mathrm{Im}(\Theta-I)$.
Conversely, if $h=(\Theta-I)(w)$ for some $w\in{T}$
then $w^{-1}tgw=tgh$.
Hence the weight orbits correspond to 
$\mathrm{Coker}(\Theta-I)\cong\mathbb{Z}$.
\end{proof}

These groups have 2-generator presentations,
but we do not know whether they have deficiency 0.

One could also construct these examples by considering torsion-free
extensions of $\beta=\pi_1^{orb}(\mathbb{D}(3,\overline{3}))$ by 
$A=\mathbb{Z}^2$ which are torsion free and have abelianization $\mathbb{Z}$.
Using the fact that the elements of finite order in $\beta$ 
must act on $A$ as in Theorem 9 leads to presentations
equivalent to those of Theorem 10.

If $G$ is embedded in the affine group $Aff(Nil)$, as in \S12 of \cite{Hi11},
then the automorphism $rk$ is induced by conjugation by the affine transformation $\Psi_w$
given by
$\Psi_w([x,y,z])=[-y+\frac13,-x-\frac13,w-z-\frac{x}3],
$ for all $[x,y,z]\in{Nil}$.
(The parameter $w$ corresponds to an element of the centre $\zeta{Nil}$.)
Since $\Psi_w$ normalizes the image of $G$ in $Aff(Nil)$,
it induces a self-homeomorphism $\psi_w$ of 
\[
N=Nil/G=M(0;(3,1),(3,1),(3, 1-3e)).
\]

\begin{cor}
If $K$ is a $2$-knot with $\pi{K}\cong\pi$ then $M(K)$ 
is a $\mathbb{N}il^3\times\mathbb{E}^1$-manifold.
No such $2$-knot is invertible or amphicheiral,
nor is it a twist spin.
\end{cor}

\begin{proof}
The group $\pi$ is the fundamental group of 
the mapping torus $N\rtimes_{\psi_w}{S^1}$, which is a
closed orientable $\mathbb{N}il^3\times\mathbb{E}^1$-manifold.
Since $\pi$ is polycyclic and of Hirsch length 4 and 
$\chi(M(K))=\chi(N\rtimes{S^1})=0$,
these manifolds are homeomorphic, 
by Theorem 8.1 of \cite{Hi}.

Parts (3) and (4) of the theorem imply that
no such $2$-knot is invertible or amphicheiral,
while no such knot is a twist spin,
since $(t^3x)^2$ is not a power of any weight element $tg$
(with $g\in\pi''$).
\end{proof}

The knot $K$ is reflexive if $M(K)$ has a self-homeomorphism which lifts to a self-homeomorphism of the mapping torus $M(\Psi_w)$ which twists the framing.
Are any of these knots reflexive?
(We suspect not. One potential complication in checking this is that the
fixed point set of $\psi_w$ is empty, for all $w\in\mathbb{R}$.)

\section{spherical and bad bases}

If the base orbifold $B$ is good and its orbifold fundamental group 
$\beta$ is finite then $B=S^2(a,a)$, $P^2(a)$, 
$\mathbb{D}(a)$, $\mathbb{D}(\overline{a},\overline{a})$,
$S(2,3,3)$, $S(2,3,4)$ or $S(2,3,5)$.
In each of the first three cases, $\beta$ is cyclic, 
while in the fourth case $\beta\cong{D_{2a}}$,
the dihedral group of order $2a$. 
In the remaining cases $\beta$ is tetrahedral,
octahedral or icosahedral, respectively.
It follows from the long exact sequence of homotopy that
the image $A$ of the fundamental group of the general fibre
in $\pi$ has rank 1.
Hence the knot group $\pi$ has two ends; 
equivalently, $\pi'$ is finite.

If $B$ is a bad 2-orbifold then $B=S^2(a)$, 
$S^2(a,b)$ or $\mathbb{D}(\overline{a},\overline{b})$, 
and $\beta$ is cyclic or dihedral.
In these cases, $\pi'$ is finite cyclic of odd order.

Every 2-knot group with finite commutator subgroup is
the group of a 2-knot $K$ such that $M(K)$ is an $\mathbb{S}^3\times\mathbb{E}^1$-manifold, and 
thus Seifert fibred.
Most, but not all, are the groups of 2-twist spins.
(See Chapters 11 and 15 of \cite{Hi}.)

\section{higher dimensions?}

Are there $n$-knots with aspherical, 
Seifert fibred knot manifold for any  $n>2$?
Even if we allow the general fibre to be 
an (orientable) infrasolvmanifold, rather than just a torus,
examples seem to be hard to find, perhaps because
one must check that $H_i(\pi;\mathbb{Z})=0$ for all $2\leq{i}\leq\lfloor\frac{n}2\rfloor+1$.
(This is feasible when $n=3$, for  if $M$ is an aspherical, orientable 5-manifold  
such that $\pi=\pi_1(M)$ is a knot group then the cocore of surgery on a 
representative of a normal generator of $\pi$ is a 3-knot in $S^5$.
Reversing the surgery shows that $M$ is a knot manifold.)

It can be shown that if  a 3-knot manifold $M$ 
is Seifert fibred (in this broader sense) 
then the base is non-orientable and the general fibre is the 3-torus
or the half-turn flat 3-manifold.

We note finally that if $n>3$ there is no algorithm to decide 
whether the groups of two orientable Seifert fibred $(n+2)$-manifolds 
(with general fibre the $n$-torus and base an orientable 2-orbifold) 
are isomorphic.
See Satz 4.8 of \cite{Zi85}.

\end{document}